\documentclass[11pt]{amsart}
\usepackage{amsmath}
\usepackage{amsthm}
\usepackage{graphicx} 
\usepackage{pxfonts}
\usepackage{txfonts}
\usepackage{url}

\theoremstyle{plain}
\newtheorem{theorem}{Theorem}
\newtheorem{e-proposition}{Proposition}

\newtheorem{conjecture}{Conjecture}
\newtheorem{question}{Question}

\newcommand{\Hy}{\mathbb{H}^2}

\newcommand{\Q}{\mathbb{Q}}

\newcommand{\R}{\mathbb{R}}

\newcommand{\SLZ}{\mathrm{SL}_2(\Z)}
\newcommand{\SLQ}{\mathrm{SL}_2(\Q)}

\newcommand{\tr}{\mathrm{tr}}
\newcommand{\T}{\mathbb{T}}

\newcommand{\U}{\mathrm{T}^1}

\newcommand{\Z}{\mathbb{Z}}

\begin{document}

\title{Almost commensurability of 3-dimensional Anosov flows}
\author{Pierre Dehornoy}
\thanks{
I thank \'Etienne Ghys for raising the questions discussed in this note and for several related conversations, and Wang Shicheng for pointing out Proposition~\ref{SWW} to me when I was visiting Beijing University. \\Supported by SNF project 137548 \emph{Knots and surfaces}.
}

\begin{abstract}
Two flows are \emph{topologically almost commensurable} if, up to removing finitely many periodic orbits and taking finite coverings, they are topologically equivalent. We prove that all suspensions of automorphisms of the 2-dimensional torus and all geodesic flows on unit tangent bundles to hyperbolic 2-orbifolds are pairwise almost commensurable.
\end{abstract}


\maketitle


\subsection{Topological equivalence for Anosov flows}
A general problem in dynamical systems is to classify flows up to \emph{topological equivalence}, that is, up to applying an homeomorphism of the underlying space that maps orbits onto orbits~\footnote{We consider in this note homeomorphisms and covering maps, so that the manifold structure is of importance. For a weaker notion and related results on hyperbolic flows, see M. Boyle's article~\cite{Boyle}.}.

Of special interest are the \emph{Anosov flows} which are, in some sense, the simplest chaotic systems~\cite{Anosov, Hadamard}. 
By definition, a flow is of Anosov type if there exists an invariant splitting of the tangent bundle of the underlying manifold into a Whitney sum of three bundles: the bundle generated by the direction of the flow, a contracting bundle along which (the differential of) the flow is uniformly contracting, and an expanding bundle along which the flow is uniformly expanding.  
In the 3-dimensional case, all three bundles have to be 1-dimensional.

There are two main examples of 3-dimensional Anosov flows, namely the vertical flow on the suspension of a hyperbolic automorphism of the torus (that is, the flow~$\partial/\partial t$ on $M_A:=\T^2\times[0,1]/_{(x,1)\sim(Ax,0)}$ where $A$ is an element of~$\SLZ$ satisfying~$\tr(A)>2$) and the geodesic flow on the unit tangent bundle to a hyperbolic 2-orbifold (that is, on~$\U\Hy/G$ for $G$ a Fuchsian group, the flow whose orbits are of the form~$(\gamma(t), \dot\gamma(t))$ for $\gamma$ a geodesic). 
For these flows, the question of topological equivalence can be answered completely: a suspension and a geodesic flow are never equivalent, the suspensions of two automorphisms are equivalent if and only if the associated matrices are conjugated in~$\SLZ$, and two geodesic flows are equivalent if and only if the underlying 2-orbifolds are of the same type (the latter statement is not obvious, it follows from the structural stability of the geodesic flow in negative curvature and from the convexity of the space of negatively curved metrics, see Ghys' article~\cite{GhysConjugues}).

\subsection{Topological almost equivalence and Birkhoff sections}

Suspensions and geodesic flows can be connected if one considers a weaker notion: two flows $\phi, \phi'$ on two manifolds $M, M'$ are \emph{topologically almost equivalent} if there exists a finite collection~$\Gamma$ (\emph{resp.} $\Gamma'$) of periodic orbits of $\phi$ (\emph{resp.} $\phi'$) such that $\phi\vert_{M\setminus\Gamma}$ is topologically equivalent to $\phi'\vert_{M'\setminus\Gamma'}$. 

G.\,Birkhoff showed~\cite{BirkhoffDS} that the geodesic flow on a hyperbolic genus~$g$ surface admits a \emph{Birkhoff section}, that is, a surface whose boundary is the union of finitely many periodic orbits and whose interior is transverse to the flow and intersects every orbit in bounded time. 
Then D.\,Fried observed~\cite{FriedAnosov} that Birkhoff's surface is of genus~1 and that the first return map is of Anosov type, which implies that the geodesic flow is almost equivalent to the suspension of some automorphism of the torus.
The latter was determined by \'E.\,Ghys~\cite{GhysGodbillon} and N.\,Hashiguchi~\cite{Hashiguchi}: it corresponds to the matrix~$\left(\begin{smallmatrix} g&g+1\\g-1&g \end{smallmatrix}\right)^2$.

In the same direction, Fried also showed~\cite{FriedAnosov} that every transitive Anosov flow admits a Birkhoff section (with no control of the genus in general) and that the associated first return map is of pseudo-Anosov type. This implies that every transitive Anosov flow is topologically almost equivalent to the suspension of some pseudo-Anosov homeomorphism. 
These observations led Fried~\cite{FriedAnosov} to ask
 
\begin{question}[Fried]
\label{Fried}
Does every transitive Anosov flow admit a genus one Birkhoff section?
\end{question}

A positive answer would imply that every transitive Anosov flow is topologically almost equivalent to the suspension of some automorphism of the torus. Some progress about this question in the case of geodesic flows was reported in~\cite{Pierre}, but no general answer is known in general. 

\subsection{Topological commensurability}
\label{Commensurability}

Topological equivalence can be weakened in another direction: two flows are called \emph{topologically commensurable} if they admit finite coverings by topologically equivalent flows. 

Since every hyperbolic 2-orbifold is finitely covered by some hyperbolic surface and since any two hyperbolic surfaces are covered by surfaces of the same genus, the geodesic flows on the unit tangent bundles of any two hyperbolic 2-orbifolds are topologically commensurable. 

For suspensions of automorphisms of the torus, there is more than one commensurability class: 

\begin{e-proposition}[Sun-Wang-Wu, \cite{SWW}]
\label{SWW}
Two hyperbolic matrices $A, B$ give rise to topologically commensurable suspensions~$M_A, M_{B}$ if and only if there exist two positive integer $i,j$ satisfying $\tr(A^i)=\tr(B^j)$.
\end{e-proposition}

\begin{proof}
We only prove the ``if'' part, and refer to Sun-Wang-Wu for the ``only if'' part, as we do not need it for our main result. First, it is clear that, if $B=A^k$ holds, then $M_{B}$ is a $k$-fold cover of~$M_A$, so that we can restrict our attention to matrices with the same trace.

Now, suppose $\tr(A)=\tr(B)$. Then, the matrices $A$ and $B$ are conjugated in~$\SLQ$, so there exists a matrix~$P$ with integer coefficients (but with determinant not necessarily equal to~$\pm1$) satisfying~$B=P^{-1}AP$. 
The canonical action of $A$ on $\R^2$ fixes the lattice~$\Z^2$, and permutes all the (finitely many) sublattices of~$\Z^2$ of any given index. 
In particular, let $\Lambda_P$ be the sublattice of~$\Z^2$ generated by~$P$. 
Then there exists an integer $k$ for which $A^k$ fixes~$\Lambda_P$, hence acts on the torus~$\Z^2/\Lambda_P$.
In the basis spanned by~$P$, this action is described by the matrix~$P^{-1}A^kP=B^k$. 
Therefore the covering~$\R^2/\Lambda_P\to\R^2/\Z^2$ induces a covering~$M_{B}\to M_{A}$ of index~$k\vert\!\det(P)\vert$. 
Moreover, the covering preserves the vertical direction.
\end{proof}

\subsection{Topological almost commensurability}

Merging the two previous weakenings of topological equivalence, one obtain a new one: two flows $\phi, \phi'$ on two manifolds $M, M'$ will be called \emph{topologically almost commensurable} if there exists a finite collection~$\Gamma$ (\emph{resp.} $\Gamma'$) of periodic orbits of $\phi$ (\emph{resp.} $\phi'$) such that $\phi\vert_{M\setminus\Gamma}$ is topologically commensurable to $\phi'\vert_{M'\setminus\Gamma'}$. 
Topological almost commensurability is an equivalence relation. 
The topological commensurability of geodesic flows and the construction of Birkhoff and Fried led Ghys to propose

\begin{conjecture}[Ghys]
\label{Ghys}
Any two transitive Anosov flows are topologically almost commensurable.
\end{conjecture}

Our main observation here is that the results of~\cite{Pierre} can be used to provide a proof in the case of geodesic flows and suspensions of the torus: 

\begin{theorem}
All suspensions of automorphisms of the 2-torus and all geodesic flows on unit tangent bundles to hyperbolic 2-orbifolds are pairwise topologically almost commensurable.
\end{theorem}

\begin{proof}
Let $G_{2,3,t+4}$ be the index~2 subgroup of the group generated by the symmetries along the edges of a hyperbolic triangle with angles~$\pi/2, \pi/3, \pi/(t{+}4)$. 
Then~$\Hy/G_{2,3,t+4}$ is a hyperbolic orbifold that is a sphere with three conic points of order~$2, 3$, and $t+4$ respectively.
Proposition~C of~\cite{Pierre} states that the geodesic flow on~$\U\Hy/G_{2,3,t+4}$ admits a Birkhoff section of genus one with one boundary component (depicted on Figure~\ref{F}), such that the first return map is conjugated to~$\left(\begin{smallmatrix} 0&1\\-1&t \end{smallmatrix}\right)$. 
This implies that the restriction of the geodesic flow to the complement of the boundary of the Birkhoff section is topologically equivalent to the vertical flow on the complement of one periodic orbit on~$M_{\left(\begin{smallmatrix} 0&1\\-1&t \end{smallmatrix}\right)}$. By definition, this means that the geodesic flow on~$\U\Hy/G_{2,3,t+4}$ is topologically almost equivalent to the vertical flow on~$M_{\left(\begin{smallmatrix} 0&1\\-1&t \end{smallmatrix}\right)}$.

\begin{figure}[ht]
\centering
\includegraphics*[width=.52\textwidth]{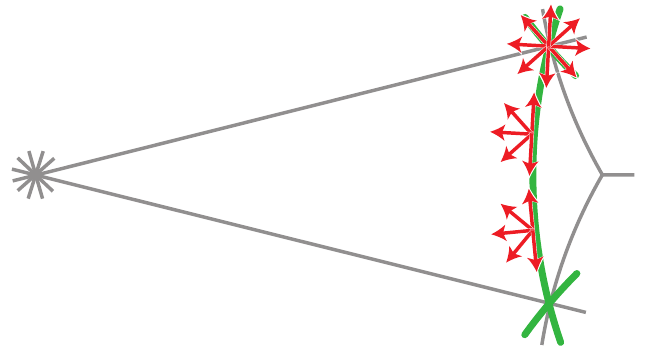}
\caption
{\small The genus one Birkhoff section for the geodesic flow on~$\U\Hy/G_{2,3,t+4}$ constructed in~\cite{Pierre} (here with $t+4=12$). A fundamental domain for~$\Hy/G_{2,3,t+4}$ is obtained by gluing two symmetric hyperbolic right-angled triangles along their longest side. The geodesic that connects the two right angles induces a closed geodesic~$\gamma$ in~$\Hy/G_{2,3,t+4}$. The Birkhoff section is made of those unit tangents to~$\gamma$ that point into one given side of~$\gamma$ (here the side of the order~$12$ vertex), plus the whole fiber of the order~2 point.}
\label{F}
\end{figure}

Now, Proposition~\ref{SWW} implies that every suspension of the torus is topologically commensurable to the suspension of~$\left(\begin{smallmatrix} 0&1\\-1&t \end{smallmatrix}\right)$ for some~$t$. 
Since any two geodesic flows are topologically commensurable, the result follows.
\end{proof}

Owing to the previous result, a proof of Ghys' conjecture would now follow from a positive answer to Fried's question. 
Nevertheless, this question seems hard for Anosov flows that are not of the type considered here. 
It looks similar to the following open question in contact geometry: ``Is every tight contact structure supported by a genus one open book?'', for which very little is known (see Etnyre and Ozbagci~\cite{Etnyre}).


\bibliographystyle{siam}

\end{document}